\pdfoutput=1
\documentclass{article}
\usepackage[latin2]{inputenc}
\usepackage {latexsym,amsmath,amsfonts} 
\usepackage {amssymb}
\usepackage {graphicx}

\title{Star chromatic index}

\author{Zden\v{e}k Dvo\v{r}\'ak 
\thanks{
  Email: {\tt rakdver@kam.mff.cuni.cz}. 
  Institute for Theoretical Computer Science is supported as project 1M0545 by Ministry of Education of the Czech Republic. 
  Partially supported by a Czech-Slovenian bilateral project MEB 091037 and BI-CZ/10-11-004.
  } 
\\[2mm] 
  KAM \& ITI \\ 
  Charles University \\ 
  Prague 
\and
  Bojan Mohar
  \thanks{
  Email: {\tt mohar@sfu.ca}. 
  Supported in part by an NSERC Discovery Grant (Canada),
  by the Canada Research Chair program, and by the
  Research Grant P1--0297 of ARRS (Slovenia).}~\thanks{On leave from:
  IMFM \& FMF, Department of Mathematics, University of Ljubljana, Ljubljana,
  Slovenia.}
\\[2mm] 
  Dept.~of Mathematics \\
  Simon Fraser University \\
  Burnaby
\and
  Robert \v{S}\'{a}mal\thanks{
  Email: {\tt samal@kam.mff.cuni.cz}. 
  Institute for Theoretical Computer Science is supported as project 1M0545 by Ministry of Education of the Czech Republic. 
  Partially supported by grant GA \v{C}R P201/10/P337.
  Partially supported by a Czech-Slovenian bilateral project MEB 091037 and BI-CZ/10-11-004.}
\\[2mm] 
  KAM \& ITI \\ 
  Charles University \\ 
  Prague 
}

\newtheorem{theorem}{Theorem}[section]
\newtheorem{lemma}[theorem]{Lemma}

\newtheorem{conjecture}{Conjecture}
\newtheorem{question}[conjecture]{Question}

\newcommand{\eopf}{\framebox[2.5mm]}
\newenvironment{proofof}[1]%
{\noindent{\bf Proof (#1).}\ }%
{\hfill\eopf\par\bigskip}%

\newenvironment{proof}%
{\noindent{\bf Proof.}\ }%
{\hfill\eopf\par\bigskip}%

\def\ceil#1{\lceil #1 \rceil}
\let\eps\varepsilon
\newcommand\chis{\chi_{\rm s}}

\begin{document}

\date{}
\maketitle

\begin{abstract}
The star chromatic index $\chis'(G)$ of a graph $G$ is the minimum number 
of colors needed to properly color the edges of the graph so that no path or cycle
of length four is bi-colored. We obtain a near-linear upper bound in terms 
of the maximum degree $\Delta=\Delta(G)$. Our best lower bound 
on $\chis'$ in terms of $\Delta$ is $2\Delta(1+o(1))$ valid for complete graphs. 
We also consider the special case of cubic graphs, 
for which we show that the star chromatic index lies between 4 and 7 and 
characterize the graphs attaining the lower bound. 
The proofs involve a variety of notions from other branches of mathematics
and may therefore be of certain independent interest.
\end{abstract}

\section{Motivation}

Edge-colorings of graphs have long tradition. Although the chromatic index of
a graph with maximum degree $\Delta$ is either equal to $\Delta$ or $\Delta+1$
(Vizing \cite{Vizing}), it is hard to decide when one or the other value
occurs. This is a consequence of the fact that distinguishing between graphs 
whose chromatic index is $\Delta$ or $\Delta+1$ is NP-hard (Holyer \cite{Holyer}).
This is true even for the special case when $\Delta=3$ (cubic and subcubic graphs).

Two special parameters concerning vertex colorings of graphs under some additional
constraints have received lots of attention. The first kind is that of an 
\emph{acyclic coloring} (see \cite{Gr,AMR}), where we ask not only that every color 
class is an independent vertex set but also that any two color classes induce an 
acyclic subgraph. The second kind is obtained when we request 
that any two color classes induce a star forest --- this variant is called 
\emph{star coloring} (see \cite{ACKKR,NO} for more details). These types of colorings 
give rise to the notions of the \emph{acyclic chromatic number} and the 
\emph{star chromatic number} of a graph, respectively. 

A \emph{proper $k$-edge-coloring} of a graph $G$ is a mapping
$\varphi: E(G)\to C$, where $C$ is a set (of \emph{colors}) of cardinality $k$,
and for any two adjacent edges $e,f$ of $G$, we have $\varphi(e)\ne \varphi(f)$.
A subgraph $F$ of $G$ is said to be \emph{bi-colored} (under the edge-coloring
$\varphi$) if $|\varphi(E(F))|\le 2$. 
A proper $k$-edge-coloring $\varphi$ is an 
\emph{acyclic $k$-edge-coloring} if there are no bi-colored cycles in $G$, and
is a \emph{star $k$-edge-coloring} if there are neither bi-colored 4-cycles
nor bi-colored paths of length 4 in $G$ (by length of a path we mean its
number of edges). 
The \emph{star chromatic index} of $G$,
denoted by $\chis'(G)$, is the smallest integer $k$ such that $G$ admits 
a star $k$-edge-coloring. 
Note that the above definition of acyclic/star edge-coloring of a graph $G$
is equivalent with acyclic/star vertex coloring of the line-graph $L(G)$.  

\newcommand\Gd{{\mathcal G}_\Delta}

If one considers the class of graphs $\Gd$ of maximum degree at most $\Delta$,
Brooks' Theorem shows that the usual chromatic number is $O(\Delta)$.
The maximum acyclic chromatic number on $\Gd$ is
$\Omega(\Delta^{4/3}/\log^{1/3}\Delta)$ and $O(\Delta^{4/3})$ 
(Alon, McDiarmid, and Reed \cite{AMR}). 
The maximum star chromatic number on $\Gd$ is 
$\Omega(\Delta^{3/2}/\log^{1/2}\Delta)$ and $O(\Delta^{3/2})$ 
(Fertin, Raspaud, and Reed \cite{FRR}). 

In contrast with the aforementioned $\Delta^{4/3}$ behaviour in the class of 
all graphs of maximum degree $\Delta$, the acyclic chromatic index is linear
in terms of the maximum degree. Alon et al.\ \cite{AMR} proved
that it is at most $64\Delta$, and Molloy and Reed \cite{MR} improved 
the upper bound to $16\Delta$. 
One would expect a similar phenomenon 
to hold for star edge-colorings. However, the only previous work \cite{Liu-Deng}
just improves the constant in the bound $O(\Delta^{3/2})$ from 
vertex coloring. 

In this paper we show a near-linear upper bound for the star chromatic index
in terms of the maximum degree (Theorem \ref{thm:upperdelta}).
Additionally, we provide some lower bounds (Theorem \ref{thm:lowerdelta})
and consider the special case of cubic graphs (Theorem \ref{thm:cubic}).
The proofs involve a variety of notions from other branches of mathematics
and are therefore of certain independent interest.

\section{Upper bound for $\chis'(K_n)$}

We shall first treat the special case of complete graphs. The study of
their star chromatic index is motivated by the results presented in 
Section \ref{sect:3} since they give rise to general upper bounds on the
star chromatic index.

\begin{theorem} 
\label{thm:upperKn}
The star chromatic index of the complete graph\/ $K_n$ satisfies 
$$
  \chis'(K_n) \le n \cdot \frac{ 2^{ 2\sqrt2(1+o(1)) \sqrt{\log n} } }{(\log n)^{1/4}} \,.
$$
In particular, for every $\eps>0$ there exists a constant $c$ such that
$\chis'(K_n) \le cn^{1+\eps}$ for every $n\ge 1$.
\end{theorem}

\begin{proof}
Let $A$ be an $n$-element set of integers, to be chosen later.
We will assume that the vertices of $K_n$ are exactly the
elements of~$A$, $V(K_n)=A$, and color the edge $ij$ by color $i+j$. 

Obviously, this defines a proper edge-coloring. 
Suppose that $ijklm$ is a bi-colored path (or bi-colored 4-cycle).
By definition of the coloring we have $i+j = k+l$ and $j+k = l+m$,
implying $i + m = 2k$. Thus, if we ensured that the set $A$ does not contain 
any solution to $i+m=2k$ with $i,m \ne k$ we would have found a star edge-coloring of $K_n$. 
It is easy to see, that such triple $(i,k,m)$ forms a 3-term arithmetic progression; 
luckily, a lot is known about sets without these progressions. 

We will use a construction due to Elkin \cite{Elkin} (see also \cite{GreenWolf} 
for a shorter exposition) who has improved an earlier result
of Behrend \cite{Behrend}. As shown by Elkin \cite{Elkin}, there is a set
$A \subset \{1, 2, \dots, N\}$ of cardinality at least $c_1 N(\log N)^{1/4}/2^{ 2\sqrt2 \sqrt{\log N} }$ 
such that $A$ contains no 3-term arithmetic progression. 

The defined coloring uses only colors $1, 2, \dots, 2N$ (possibly not all of them), 
thus we have shown that $\chis'(K_n) \le 2N$.
We still need to get a bound on $N$ in terms of $n$. 
In the following, $c_1,c_2,\dots$ are absolute constants. 

For every $\eps > 0$ we have
\begin{equation}
 n = |A| \ge c_1 N \frac{ (\log N)^{1/4} }{ 2^{ 2\sqrt2 \sqrt{\log N} } } 
         \ge c_2 N^{1-\eps} 
 \label{eq:A}
\end{equation}
Since we also have $N \le c_3 n^{1+\eps}$ for every $\eps > 0$, we may plug this in
(\ref{eq:A}) and use the fact that 
$(\log N)^{1/4}\, 2^{-2\sqrt2 \sqrt{\log N}}$ 
is a decreasing function of~$N$ for large $N$ to conclude that
$$
  n \ge c_4 N ((1+\eps)\log n)^{1/4}\ 2^{-2\sqrt2 \sqrt{(1+\eps)\log n}}\,.
$$
Thus we get 
$\displaystyle N \le c_5 n 2^{2\sqrt2 \sqrt{(1+\eps)\log n}}\,(\log n)^{-1/4}$. 
One more round of this `bootstrapping' yields the desired inequality
$$
  N \le n \frac{ 2^{ 2\sqrt2(1+o(1)) \sqrt{\log n} } }{(\log n)^{1/4}}\,.
$$
\end{proof}

\paragraph{Remark.} A tempting possibility for modification
is to use a set $A$ (in an arbitrary group) that contains
no 3-term arithmetic progression and $|A+A|$ is small.
Any such set could serve for our construction, with the same proof.
Even more generally, we only need a symmetric function $p:A\times A\to N$, 
where $A=\{1,2,\dots,n\}$, such that $p(a,\cdot): A\to N$ is a 1-1 function
for each fixed $a\in A$, $N$ is small, and $p$ does not yield bi-colored
paths (for all $i,j,k,l,m$ we either have $p(i,j)\ne p(k,l)$ or $p(j,k)\ne p(l,m)$).
We have been unable, however, to find a set that would yield
a better bound than that of Theorem \ref{thm:upperKn}.

\section{An upper bound for general graphs}
\label{sect:3}

The purpose of this section is to present a way to find 
star edge-coloring of an arbitrary graph~$G$, using 
a star edge-coloring of the complete graph $K_n$ with
$n = \Delta(G)+1$. 

We will use the concept of \emph{frugal colorings} as defined by 
Hind, Molloy and Reed \cite{HMR}. 
A proper vertex coloring of a graph is called \emph{$\beta$-frugal} if no more 
than $\beta$ vertices 
of the same color appear in the neighbourhood of a vertex.
Molloy and Reed \cite{MR,MR2} proved that every graph has an
$O(\log \Delta/\log\log \Delta)$-frugal coloring using $\Delta+1$ colors.
If $\Delta$ is large enough, one may use $50$ for the implicit constant
in the $O(\log \Delta/\log\log \Delta)$ asymptotics.

\begin{theorem} \label{thm:upperdelta}
For every graph $G$ of maximum degree $\Delta$ we have 
\begin{equation}
  \chis'(G) \le \chis'(K_{\Delta+1}) \cdot 
       O\Bigl(\frac{\log \Delta}{\log\log \Delta}\Bigr)^2
\label{eq:B}
\end{equation}
and therefore\/ $\chis'(G) \le \Delta \cdot 2^{O(1) \sqrt{\log \Delta}}$. 
\end{theorem}

\begin{proof}
Using the above-mentioned result of Molloy and Reed \cite{MR2}, we find 
a $\beta$-frugal $(\Delta+1)$-coloring $f$ with 
$\beta = O(\log \Delta/\log\log \Delta)$.
We assume the colors used by $f$ are the vertices of $K_{\Delta+1}$, 
so that the frugal coloring is $f: V(G) \to V(K_{\Delta+1})$. 
Let $c$ be a star edge-coloring of $K_{\Delta+1}$.
A natural attempt is to color the edge $uv$ of $G$ by $c(f(u)f(v))$. 
This coloring, however, may not even be proper: if a vertex $v$ has 
neighbours $u$ and $w$ of the same color, then the edges $vu$ and $vw$ will
be of the same color. To resolve this, we shall produce another edge-coloring,
with the aim to distinguish these edges; then we will combine the two
colorings. 

We define an auxiliary coloring $g$ of $E(G)$ using $2\beta^2$ 
colors. Let us first set 
$$
  V_i = \{ v \in V(G) : f(v) = i \}, \qquad i \in V(K_{\Delta+1})
$$
and define the induced subgraphs $G_{ij} = G[V_i \cup V_j]$. 
For each pair $\{i,j\}$ we shall define the coloring $g$ on the edges
of $G_{ij}$; in the end this will define $g(e)$ for every
edge $e$ of $G$. Recall that the frugality of $f$ implies
that the maximum degree in $G_{ij}$ is at most $\beta$. Consequently, 
the maximum degree in the (distance) square of $L(G_{ij})$ is at most 
$2\beta(\beta-1) < 2\beta^2$. Therefore, we can find a coloring 
of $E(G_{ij})$ using $2\beta^2$ colors so that no two edges of this graph 
have the same color, if their distance in the line graph is 1 or 2.

Now we can define the desired star edge-coloring of $G$: 
we color an edge $uv$ by the pair 
$$
  h(uv) = (c(f(u)f(v)), g(uv)) \,.
$$

First, we show this coloring is proper. Consider adjacent
(distinct) edges $vu$ and $vw$. If $f(u) \ne f(w)$, then $f(u)f(v)$
and $f(v)f(w)$ are two distinct adjacent edges of $K_{\Delta+1}$, 
hence $c$ assigns them distinct colors. 
On the other hand, if $f(u)=f(w) = i$ (say), we put $j=f(v)$
and notice that $uv$ and $vw$ are two adjacent edges
of $G_{ij}$, hence the coloring $g$ distinguishes them. 

It remains to show that $G$ has no 4-path or cycle colored
with two alternating colors. Let us call such
object simply a \emph{bad path} (considering $C_4$ as a closed path). 
Suppose for a contradiction that the path $uvwxy$ is bad. 
By looking at the first coordinate of $h$ we observe
that the $c$-color of the edges of the trail $f(u)f(v)f(w)f(x)f(y)$
assumes either just one value or two alternating ones. 
As $c$ is a star edge-coloring of $K_{\Delta+1}$, 
this trail cannot be a path (nor a 4-cycle).
A simple case analysis shows that in fact $f(u)=f(w)=f(y)$
and $f(v)=f(x)$. Put $i = f(u)$, $j=f(v)$ and consider
again the $g$ coloring of $G_{ij}$. By construction,
$g(uv) \ne g(wx)$, showing that $uvwxy$ is not a bad
path, a contradiction. 
\end{proof}

As we saw in this section, an upper bound on the star chromatic
index of $K_n$ yields a slightly weaker result for general 
bounded degree graphs. We wish to note that, if convenient, one
may start with other special graphs in place of $K_n$, in particular with $K_{n,n}$. 
It is easy to see that 
$$
   \chis'(K_{n,n}) \le \chis'(K_n) + n 
$$
(if the vertices of $K_{n,n}$ are $a_i, b_i$ ($i=1, \dots, n$) 
then we color edges $a_ib_j$ and $a_jb_i$ using the color
of the edge $ij$ in $K_n$, while each edge $a_ib_i$ gets a unique color).
On the other hand, a simple recursion yields an estimate
$$
 \chis'(K_n) \le \sum_{i=1}^{\ceil{\log_2 n}} 2^{i-1} \chis'(K_{ \ceil{n/2^i}, \ceil{n/2^i}})   \,.
$$
From this it follows that 
if $\chis'(K_{n,n})$ is $O(n)$ (or $n (\log n)^{O(1)}$, $n^{1+o(1)}$, respectively) 
then $\chis'(K_{n})$ is $O(n \log n)$ (or $n (\log n)^{O(1)}$, $n^{1+o(1)}$, respectively).

\section{A lower bound for $\chis'(K_n)$}

Our best lower bound on $\chis'(K_n)$ is provided below and is linear in terms
of $n$. The upper bound from Theorem \ref{thm:upperKn} is more than
a polylogarithmic factor away from this. So, even the asymptotic behaviour of
$\chis'(K_n)$ remains a mystery.

\begin{theorem} \label{thm:lowerdelta}
The star chromatic index of the complete graph\/ $K_n$ satisfies 
$$
  \chis'(K_n) \ge 2n(1+o(1)).
$$ 
\end{theorem}

\begin{proof}
Assume there is a star edge-coloring of $K_n$ using $b$ colors.
%
%
Let $a_i$ be the number of edges of color $i$, 
let $b_{i,j}$ be the number of 3-edge paths colored $i,j,i$.
We set up a double-counting argument. Note that all sums over
$i$, $j$ are assumed to be over all available colors
(that is, from 1 to $b$). As every edge gets one color, we have
\begin{equation}
  \sum_i a_i = \binom n2 \,.
\label{eq:cc}
\end{equation}
Fixing $i$, we have a matching $M_i$ with $a_i$ edges and each 
edge sharing both ends with an edge from $M_i$ contributes to 
some $b_{i,j}$. Consequently, 
\begin{equation}
  \sum_j b_{i,j} = 4\binom {a_i}2 \,.
\label{eq:aa}
\end{equation}
Finally, we fix color $j$ and observe that each 3-edge
path colored $i,j,i$ (for some~$i$) uses two edges among
the $2a_j \cdot (n-2a_j)$ edges connecting a vertex of $M_j$
to a vertex outside of $M_j$. This leads to
\begin{equation}
  \sum_i b_{i,j} \le a_j (n-2a_j) \,.
\label{eq:bb}
\end{equation}
Now we use (\ref{eq:aa}) and (\ref{eq:bb}) to evaluate the double sum 
$\sum_{i,j} b_{i,j}$ in two ways, getting
$$
  4 \sum_i \binom {a_i}2 \le  \sum_j a_j(n-2a_j) \,.
$$
This inequality reduces to
$$
  4 \sum_i a_i^2 \le (n+2) \sum_i a_i \,.
$$
By the Cauchy-Schwartz inequality, $(\sum a_i)^2\le b\cdot \sum a_i^2$,
and then using (\ref{eq:cc}), we obtain 
$$
   4 \binom{n}{2} \le b(n+2).
$$
Therefore, $b\ge 2 n(n-1)/(n+2) = (2+o(1))n$.
\end{proof}

\section{Subcubic graphs}

A regular graph of degree three is said to be \emph{cubic}. A graph of maximum degree at most three is \emph{subcubic}.
A graph $G$ is said to cover a graph $H$ if there is a graph homomorphism from $G$ to $H$ that is locally bijective.
Explicitly, there is a mapping $f:V(G) \to V(H)$ such that whenever $uv$ is an edge of $G$, the image
$f(u) f(v)$ is an edge of $H$, and for each vertex $v \in V(G)$, $f$ is a bijection between 
the neighbours of $v$ and the neighbours of $f(v)$. 

\begin{theorem} 
\label{thm:cubic}
{\rm (a)} If\/ $G$ is a subcubic graph, then $\chis'(G) \le 7$. 

{\rm (b)} If\/ $G$ is a simple cubic graph, then $\chis'(G) \ge 4$, and the equality holds if and only if\/ $G$ covers the graph of the $3$-cube.
\end{theorem}

For the part (a) of this theorem we will need the following lemma. It seems to be possible 
to use this lemma for other classes of graphs, therefore it might be of certain independent interest. 

\begin{lemma} 
\label{thm:recursive}
Let $f:E(G) \to \{1, \dots, k\}$ be a $k$-edge-coloring.

{\rm (a)} Let $e$ be an edge of $G$. Suppose that the restriction of $f$ to $E(G)\setminus\{e\}$ 
is a star edge-coloring of $G - e$ and that $f(e)$ is distinct from $f(e')$ 
whenever $d(e,e') \le 2$ (that is, either $e, e'$ share a vertex, or a
common adjacent edge). 
Then $f$ is a star edge-coloring of $G$. 

\bigskip

{\rm (b)} Let $A$ be a set of vertices of $G$, let $B = V(G) \setminus A$, 
and let $X$ be the set of edges with one end in $A$ and the other in $B$. 
Suppose that 
\begin{enumerate}
\setlength{\itemsep}{1pt} \setlength{\parskip}{0pt} \setlength{\parsep}{0pt}
\item (a restriction of) $f$ is a star edge-coloring of\/ $G[A]$; 
\item (a restriction of) $f$ is a star edge-coloring of\/ $G[B]$; 
\item no edges $e_1, e_2$ in $X$ share a common vertex in $A$ or a common adjacent edge in $G[A]$; 
\item for every edge $e \in X$ and every edge $e'$ in $G[B] \cup X$ such that 
    $d(e,e') \le 2$ we have $f(e) \ne f(e')$
    (distance is measured in $G[B] \cup X$, not in $G$);
\item for every edge $e \in X$ and every edge $e'$ in $G[A]$ we have $f(e) \ne f(e')$.
\end{enumerate}
Then $f$ is a star edge-coloring of $G$. 
\end{lemma}

\begin{proofof}{of the lemma}
(a) Since $f$ is a star edge-coloring of $G-e$, no 4-path (or 4-cycle) in $G-e$ is bi-colored. 
If $P$ is a bi-colored 4-path (4-cycle)
containing $e$, then $P$ contains an edge of the same color as $e$ at distance $\le 2$ from $e$, 
a contradiction. 

(b) Conditions (3), (4), (5) imply that for every edge $e \in X$ and every edge $e'\in E(G)$, 
if $d(e,e')\le 2$, then $f(e) \ne f(e')$. Therefore, we can repeatedly apply part (a), starting
with the graph $G[A] \cup G[B]$ and adding one edge of $X$ at a time. 
\end{proofof}

To explain a bit the conditions of part (b) in the above lemma: the point here is that
in the condition 5, we do not check what is the distance of $e$ and $e'$. In our applications, 
$A$ will be a particular small subgraph of $G$ (such as those in Figure~\ref{fig:auxreduc}) 
and $B$ the `unknown' rest of the graph. We do not
want to distinguish whether some edges in $X$ share a vertex in $B$. This, however, may create
new 4-paths, henceforth the particular formulation of this lemma. 

\begin{proofof}{of the theorem}  
(a) Trying to get a contradiction, let us assume that $G$ is a subcubic
graph with the minimum number of edges for which $\chis'(G)>7$. 
We first prove several properties of $G$ (connectivity, absence of various
small subgraphs). This will eventually allow us to construct the desired 
7-edge-coloring by decomposing $G$ into a collection of cycles connected by paths of length 1 or~2. 

Clearly, \textbf{$G$ G is connected.} 
\textbf{Suppose that $G$ contains a cut-edge $xy$.} Let $G_x$ and
$G_y$ be the components of $G-xy$ which contain the vertex $x$ and $y$,
respectively. By the minimality of $G$, each of $G_x$ and $G_y$ admits a
star 7-edge-coloring. In $G_x$ there are at most 6 edges that are 
incident to a neighbor of $x$. By permuting the colors, we may assume that
color 7 is not used on these edges. Similarly, we may assume that color 7
is not used on the edges in $G_y$ that are incident with neighbors of $y$.
Then we can color the edge $xy$ by using color 7 and obtain a star
7-edge-coloring of $G$ (we use Lemma~\ref{thm:recursive}(a)). 
This contradiction shows that \textbf{$G$~is 2-connected.}

If $G$ contained \textbf{a path $wxyz$, where $x$ and $y$ are degree~2 vertices,} 
then we could color $G-xy$ by induction, and extend the coloring to a star
7-edge-coloring of $G$ by using Lemma~\ref{thm:recursive}(a). (For the 
edge $e=xy$ we use a color that does not appear on the at most six edges incident to $w$ or $z$.) 
Thus, \textbf{such path $wxyz$ does not exist.} In particular, $G$~is not a cycle. 

Suppose next that $G$ contains \textbf{a degree~2 vertex $z$ whose neighbors $x$ and $y$ are adjacent.} 
We will use Lemma~\ref{thm:recursive}(a) for $e=xz$. By induction we may find
a star edge-coloring of $G-e$, and as there are at most six edges in $G$ at
distance $\le 2$ from $e$, we can extend the coloring to $e$ to satisfy 
the condition of the lemma. 
So the graph $G$ can be star edge-colored using 7 colors, a contradiction. 
This shows that \textbf{the neighbors of a degree~2 vertex cannot be adjacent in $G$.} 

Further suppose that $G$ contains \textbf{parallel edges.} Three parallel edges 
would constitute the whole (easy to color) graph, so suppose there are 
two parallel edges between vertices $u$ and $v$. 
Unless $G$ contains a bridge, or $G$ has at most three vertices (and is easy to color), 
there are neighbors $u'$ of $u$, $v'$ of $v$ and $u' \ne v'$. 
By induction we can color $G \setminus \{u,v\}$. Next, we extend this coloring to 
the edges $uu'$, $vv'$, so that each of them has different color than the $\le 6$
edges at distance $\le 2$ from it. Now we distinguish two cases.
If $uu'$ and $vv'$ have different colors, say $a$ and $b$, then it is enough to use 
on the two parallel edges any two distinct colors that are different from $a$ and~$b$. 
If $uu'$ and $vv'$ have the same color, then there are at most 5 colors of edges 
at distance $\le 2$ from the parallel edges, so we may again use Lemma~\ref{thm:recursive}, part (a). 
So, \textbf{$G$ does not contain parallel edges.} 

Next we suppose that $G$ contains \textbf{one of the first three graphs in Figure~\ref{fig:auxreduc}} 
as a subgraph, where other edges of $G$ attach only 
at the vertices denoted by the empty circles, and some of these vertices may be identified. 
We use Lemma~\ref{thm:recursive}, part (b). We let $A$ be the set of 
vertices of the subgraph in the figure that are denoted by full circles,
so $X$ is the set of the three thick edges. 
By induction, $G[B]$ is star 7-edge-colorable. This coloring can be extended
to $G[B] \cup X$ so that color of each edge $e$ in $X$ differs from the color of all
edges at distance $\le 2$ from $e$ (there are at most 6 such edges). 
We assume that the colors used on $X$ are in $\{5, 6, 7\}$. 
For $G[A]$ we use the coloring as shown in the figure. This satisfies conditions
of Lemma~\ref{thm:recursive}, part (b), and therefore $G$ can be
star 7-edge-colored. 

Next suppose that $G$ contains \textbf{the fourth graph in Figure~\ref{fig:auxreduc}} 
as a subgraph (again, other edges can only attach at the `empty' vertices, 
some of which may be identified). 
We use Lemma~\ref{thm:recursive}, part (b) to show that $G-e$ is 
7-edge-colored in a particular way that allows us to use 
Lemma~\ref{thm:recursive}, part (a) to extend the coloring on $e$. 
We let $A$ be the vertices of the pentagon, so that $X$ is the
set of the three thick edges. Note that the conditions of the part (b) are 
satisfied for $G-e$, but not for $G$ itself. 
By induction there is a star 7-edge-coloring of $G[B]$, and we
again extend it to $X$ so that the edges in $X$ have distinct color
from edges at distance $\le 2$. Observe that there are at most six
edges in $G[B] \cup X$ that are at distance $\le 2$ from $e$, 
so there is a color, say $C$, not used on any of those. 
We shall reserve $C$ to be used at $e$. First, however, 
we apply part~(b) to color the graph~$G-e$. We use the coloring of $G[A]$
shown in the figure, assuming that $C \not\in \{1,2,3\}$ and 
that none of the colors $1,2,3$ is used on $X$. 
Finally, we use part~(a) to extend the coloring on $G$, 
letting the color of $e$ be $C$. 

As the last reduction, we show that \textbf{$G$ does not contain a path $wxyz$, 
where $w$ and $z$ are degree~2 vertices.} 
If $G$ did contain such a path, 
we could color $H = G - \{w, x, y, z\}$. Next, we describe how to extend 
this coloring to a star 7-edge-coloring of~$G$. We will denote the edges 
as in Figure~\ref{fig:pathreduc}; to ease the notation we will use $a$ to 
denote both the edge and its color. 

We may assume that all vertices among $w, x, y, z$, and their neighbours
are distinct (*), as otherwise $G$ contains one of the previously 
handled subgraphs --- those in Figure~\ref{fig:auxreduc}, triangle with a degree~2 vertex, 
parallel edges or two adjacent degree~2 vertices (the straightforward checking is left to the reader). 
It may, however, happen that, e.g., edges $s$ and $u$ have an edge adjacent 
to both of them. This has no effect on the proof, we only will have, say, $b=e$. 

The edges $a$, \dots, $h$ are 
part of $H$, so they are colored already. Similarly as in the previous cases, 
we choose a color for $s$, $t$, $u$, $v$ so that none of these edges shares a color
with an edge of $G-\{p,q,r\}$ at distance at most~2; 
using 7 colors, this is easy to achieve. 
Condition~(*) in the previous paragraph 
implies, that every 3-edge path starting at $w$, $x$, $y$~or~$z$ by an edge~$s$, $t$, $u$~or~$v$
avoids edges $p$, $q$, $r$ --- consequently, no such path has first and last edge of the same color, 
and no such path can be part of a bi-colored 4-path or 4-cycle. This greatly reduces the number 
of path and cycles we need to take care of. 

Next, we pick a color for $q$ that differs from $c$, $d$, $e$, $f$, $t$, and $u$. 

\begin{figure}[h]
\centerline{
 \includegraphics[scale=1,page=1]{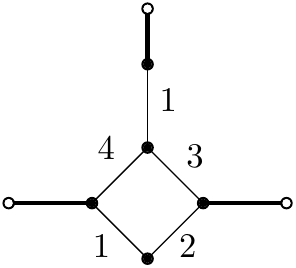}
 \hskip -1cm
 \includegraphics[scale=1,page=2]{figauxreduc}
 \hskip -1.5cm
 \includegraphics[scale=1,page=4]{figauxreduc}
 \includegraphics[scale=1,page=3]{figauxreduc}
}
\caption{Subgraphs that cannot appear in a minimal counterexample.}
\label{fig:auxreduc}
\end{figure}

\begin{figure}[h]
\centerline{\includegraphics[scale=1]{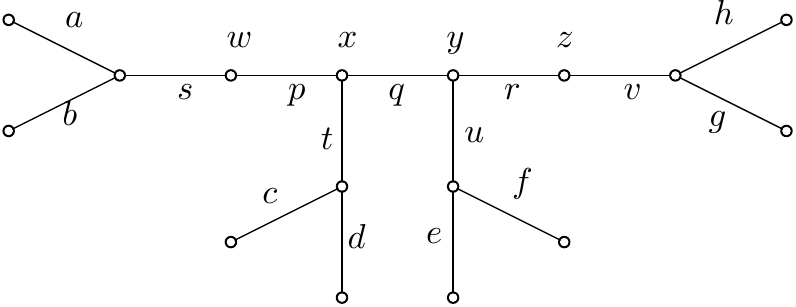}}
\caption{Illustration of the proof that minimal counterexample to Theorem~\ref{thm:cubic}
does not contain path $wxyz$ as depicted in the figure.}
\label{fig:pathreduc}
\end{figure}

Now, we distinguish several cases based on colors of $s$, $q$, and $v$. 
We again assume the colors are 1, \dots, 7; up to symmetry we have only the
following cases. 

\textbf{Case 1. } $s=1, q=2, v=3$  \\
We only need to avoid bi-colored paths 
$aspt$, $bspt$, $sptc$, $sptd$, and the four symmetrical paths
in the right part of the figure. 
If $t=1$, we choose $p$ to be different from $1, 2, a, b, c, d$.
If $t \ne 1$, it suffices to make $p$ different from $s, q, t$. 
The procedure for $r$ is analogous. 

\textbf{Case 2. } $s=1, q=1, v=2$  \\
In this case $t\ne 1$, so we only need to avoid bi-colored paths
$aspq, bspq, spqr, spqu$ and 
$urvh, urvg$, $eurv, furv$. 
If $u=2$, we make sure that $r$ differs from 
$1,2,e,f,g,h$. Otherwise, it suffices to make
$r$ different from $q, u, v$. 
Then we choose $p$ to differ from 
$a, b, 1, t, u, r$. 

\textbf{Case 3. } $s=1, q=2, v=1$  \\
This is handled in exactly the same way as Case 1.

\textbf{Case 4. } $s=1, q=1, v=1$  \\
Now $t, u \ne 1$, so we only need to avoid the paths 
$aspq, bspq, spqr, pqrv, qrvg, qrvh$, and 
$spqu, tqrv$. 
To do this, we only need to ensure, that 
$p \ne 1, a, b, t, u, r$ and 
$r \ne 1, h, g, u, t, p$, which is easily possible. 
This finishes the proof of the claim that minimal counterexample $G$ does not contain
a path $wxyz$, where $w$ and $z$ are degree~2-vertices. 

\bigskip

This finished the first part of the proof. 
Next we will use the above-derived properties of the supposed minimal 
counterexample~$G$ to find its star 7-edge-coloring and thus reach a contradiction. 
We will use only the boldface claims from the above part of the proof. 

Let $G'$ be the graph obtained from $G$ by suppressing all degree~2
vertices, i.e., replacing each path $xzy$, where $z$ is a degree~2 vertex,
by a single edge $xy$. Clearly, $G'$ is a cubic graph. It is bridgeless
(as $G$ is bridgeless) and contains no parallel edges -- as $G$~contains 
no parallel edges, no triangle with a degree~2 vertex and no 4-cycle 
with two opposite degree~2 vertices. 

By a result of Kaiser and \v{S}krekovski \cite{KS}, $G'$ contains a perfect matching~$M'$ 
such that $M'$ does not contain all edges of any minimal 3-cut or 4-cut. 
Note that each edge in $G'$ corresponds
either to a single edge in $G$ or to a path of length two. 
Let $M$ denote the set of edges of $G$ corresponding to an edge of~$M'$. 
Our goal is to use four colors (say 4, 5, 6, 7) on $M$, 
and three colors (say 1, 2, 3) on the other 
edges that form a disjoint union of circuits. 
We form an auxiliary graph~$K$, whose vertices are the edges in $M$. 
We make two of these edges $e$, $f$ adjacent in~$K$ if 
either they form a 2-edge path corresponding to an edge in~$M'$ 
or there is an edge in $G$ joining an end of $e$ with an end of~$f$. 
Observe that $K$ is a graph of maximum degree at most four. Also note that if~$K$
is disconnected, then each component contains a vertex of degree at most
three. By the Brooks Theorem, $K$ is 4-colorable unless it contains a
connected component isomorphic to $K_5$. It is easy to see that the latter
case occurs if and only if $K=K_5$ and $G=G'$.

Let us first consider the case when $K$ is 4-colorable. 
In this case we will not need the fact that $M'$ does not contain minimal 3-cuts or 4-cuts. 
The 4-coloring of the vertices of $K$ determines a 4-coloring of the edges in $M$ with the
property that every color class is an induced matching in $G$. We shall
show that we can star 3-color the edges in $G-M$ unless $G-M$ contains
a 5-cycle; this case will be treated separately. By extending that 3-edge
coloring to a 7-edge-coloring of $G$ (by using the 4-coloring of edges in
$M$) we obtain a star 7-edge-coloring since none of the four colors used on
the edges in $M$ can give rise to a bi-colored 4-path or a cycle (Lemma~\ref{thm:recursive}(a)). 
Thus it suffices to find a star 3-edge-coloring of $G-M$. This is not hard 
unless $G-M$ contains a 5-cycle. 
Recall that $G-M$ is the union of disjoint cycles and every $k$-cycle, where $k\ne 5$,
admits a star 3-edge-coloring: This is easy if $k\in\{3,4\}$ or if $k$ is
divisible by three. If $k\equiv 1\pmod{3}$ and $k>5$, we can use the colors
in the following order $1232123\cdots 123$. Similarly, if 
$k\equiv 2\pmod{3}$ and $k>5$, we can use the colors $12132123\cdots 123$. 
Thus, the only problems are the 5-cycles in $G-M$. 
To color them, we shall choose an edge $e=e_C$ in each 5-cycle $C$ and 
a color $c=c_C$, that is otherwise used as a color for $M$. 
Then we color $e$ with color $c$ and color the 4-path $C-e$ as $1,2,3,1$. 
We pick $c$ and $e$ in such a way that no edge of $M$ at distance 
at most 2 from $e$ has color $c$ (we will show below that this is possible). 
It is easy to check that this, together with the fact that $K$ is
properly colored, prevents all 4-paths and 4-cycles from being bi-colored
(Lemma~\ref{thm:recursive}(a) again). 
So, this finishes the proof of the case when $K$~is 4-colorable---provided
we show how to pick $e$ and $c$ for each 5-cycle $C$.  
To do this, we let $F$ be the set of edges of $M$ that are incident with a vertex
of $C$ but not part of $C$. 
Further, we let $X$ be the (possibly empty) set of edges of $M$ adjacent 
with some edge of $F$. Easily, $|X|$ is the number of 2-edge paths
in $M$ that are adjacent to $C$. (A 2-edge path with both ends at $C$
counts twice; in this case $X$ and $F$ intersect.)  
As $G$ contains no 3-edge path with both ends of degree 2, we have $|X| \le 2$. 
We distinguish two cases based on the color pattern on edges of $F$. 
These cases cover all possibilities up to renaming the colors. 

\textbf{Case 1.} Edges of $F$ use in some order colors 4, 4, 5, 6, and 7 
(that is, one color appears twice, the other colors once). If $X=\emptyset$, there
are three possible choices for edge $e$: for each color $c$ among 5, 6, and 7 
we may choose the edge of $C$ opposite to the edge of $F$ colored $c$. 
Edges of $X$ may be at distance 2 to some of these edges of $C$. However, 
there are at most two such edges, hence at most two colors are affected. 
So, one of colors 5, 6, and 7 is still valid. 

\textbf{Case 2.} Edges of $F$ use in some order colors 4, 4, 5, 5, and 6 
(that is, two colors twice, one once, one not at all).
In this case, if $X=\emptyset$, all five edges of $C$ can be colored 7. 
Each edge of $X$ (if such an edge exists) is at distance 2 from two edges of $C$, 
so one edge of $C$ is far from edges of $X$ and we can let this edge be $e$
and $c$ be 7. 

\bigskip

Finally, let us consider the case when $K$ does not admit a 4-coloring,
i.e., $K=K_5$. As argued before, this implies that $G=G'$ is a cubic graph
containing precisely 10 vertices. Note that $G-M$ is a 2-regular graph 
with no 3-cycles or 4-cycles (due to the choice of $M'$). 
Thus $G-M$ is isomorphic either to a 10-cycle, or to the union of two 5-cycles. 

If $G-M$ is the union of two
5-cycles, then it is easy to check that $G$ is the Petersen graph, and hence
$\chis'(G)=5$. (A star 5-edge-coloring is easy to find, and the star
4-edge-coloring does not exist as shown in part (b) below.)

The final case is that $G-M$ is a 10-cycle. 
Color its chords with colors~1, 2, 3, 4,~5. Then color the first, fourth, and seventh edges 
by colors from 1, 2, 3, 4, 5 so that no two edges at distance two share a color. 
Finally, color the remaining edges with 6~and~7.
This completes the proof.

\bigskip
(b)
Every 3-edge-coloring of a cubic graph has bi-colored cycles, thus
$\chis'(G)\ge4$. 
In Figure~\ref{fig:cube} there is a 4-edge-coloring of the cube $Q_3$.
It is easy to verify that this is indeed a star edge-coloring. 
Perhaps the fastest way to see this is to observe that for each $i\ne j$, there is
(a unique) 3-edge path colored $i,j,i$ between the two vertices colored~$j$. 
Consider now a graph $G$ that covers $Q_3$ and use the covering map to 
lift the edge-coloring of~$Q_3$ to an edge-coloring of~$G$. 
From the definition of covering projections we see that a path of length $2$ in $G$ is mapped 
to a path of length $2$ in $Q_3$. It follows that the defined edge-coloring is
proper. It also follows that a path of length $4$ in $G$ is mapped to a path
of length $4$ in $Q_3$ or to a $4$-cycle in $Q_3$, and a 4-cycle in $G$ is 
mapped to a 4-cycle in $Q_3$. It follows that we have a star edge-coloring of $G$. 

For the reverse implication, suppose that $G$ has a star 4-edge-coloring $c$.
Let us first define a (vertex) 4-coloring $f$ by letting 
$f(v)$ be the (unique) color that is missing on edges incident with $v$. 

\textbf{$f$ is a proper coloring.} For a contradiction, suppose that $f(u)=f(v)$ 
for an edge $uv$ of $G$. Let $u_1$, $u_2$ be the other neighbors of $u$, 
and $v_1$, $v_2$ be the other neighbors of $v$. 
By symmetry we may assume that $f(u)=f(v)=4$, $c(uv)=3$, 
$f(uu_i) = f(vv_i)=i$ (for $i=1,2$).
The bi-chromatic paths $u_i u v v_i$ imply that $3$ is neither used
on edges incident with $v_1$ nor on those incident with $v_2$. This, however, implies
that there is an edge-colored $2$ incident with $v_1$ and 
an edge-colored $1$ incident with $v_2$, which create a bi-chromatic
4-edge path (or 4-cycle), a contradiction. 
Note that the cases where $uv$ is contained in a triangle
($u_1=v_2$, $u_2=v_1$ or both) are also covered by the above. 

\textbf{$f$ is a covering map $G \to K_4$.} Suppose for a contradiction
that there is a vertex $v$ with neighbors $v_1$, $v_2$ such that
$f(v_1) = f(v_2)$. By symmetry we may assume that $f(v)=4$, 
$f(v_1)=f(v_2)=3$, $c(vv_i) = i$. 
Now $v_1$ must be incident with an edge of color~$2$
and $v_2$ must be incident with an edge of color~$1$, 
producing again a bi-chromatic 4-edge path (or cycle).

\textbf{$f$ together with $c$ define a covering $G \to Q_3$.} 
Let $i,j,k,l$ denote $1,2,3,4$ in some order. If a vertex $v$ of $G$
has $f(v)=i$ then the $c$-colors of its incident edges are $j$, $k$, $l$ 
and the same holds for the $f$-colors of its adjacent vertices. 
There are exactly two possibilities: either the edges incident with $v$ colored $j,k,l$ 
lead to vertices colored $k,l,j$ (respectively), or to vertices colored $l,j,k$. 
We refer to these two possibilities as the \emph{local color pattern} at $v$. 

Observe that in $Q_3$ as depicted in Figure~\ref{fig:cube}, 
there are for each $i$ two vertices colored $i$ and they use different local 
color patterns. 
This implies there is a unique vertex mapping $F: V(G) \to V(Q_3)$ such that for each 
$v \in V(G)$ the following conditions hold:
\begin{enumerate}
\setlength{\itemsep}{1pt} \setlength{\parskip}{0pt} \setlength{\parsep}{0pt}
\item we have $f(v) = f(F(v))$ (we use $f$ also for the vertex coloring
      of $Q_3$), and 
\item $v$ and $F(v)$ use the same local color pattern. 
\end{enumerate}

To show that $F$ is a covering map, we need to observe that for each $v\in V(G)$, 
the three neighbours of $v$ in $G$ map by $F$ to the three neighbours of $F(v)$ in $Q_3$. 
As we already know that $f$ is a covering map to $K_4$, it suffices to show, that
a neighbour $u$ of $v$ is indeed mapped to the neighbour of $F(v)$ with color $f(u)$
(and not to the other vertex with the same color). 
For this we observe that the local coloring pattern at a vertex $v$ determines the 
local coloring pattern at each neighbour of $v$, in any cubic graph that is 
star 4-edge-colored. As this holds both in $G$ and in $Q_3$, 
our definition of $F$ yields a covering map, which finishes the proof. 
\end{proofof}

\begin{figure}
\centerline{\includegraphics[scale=0.5]{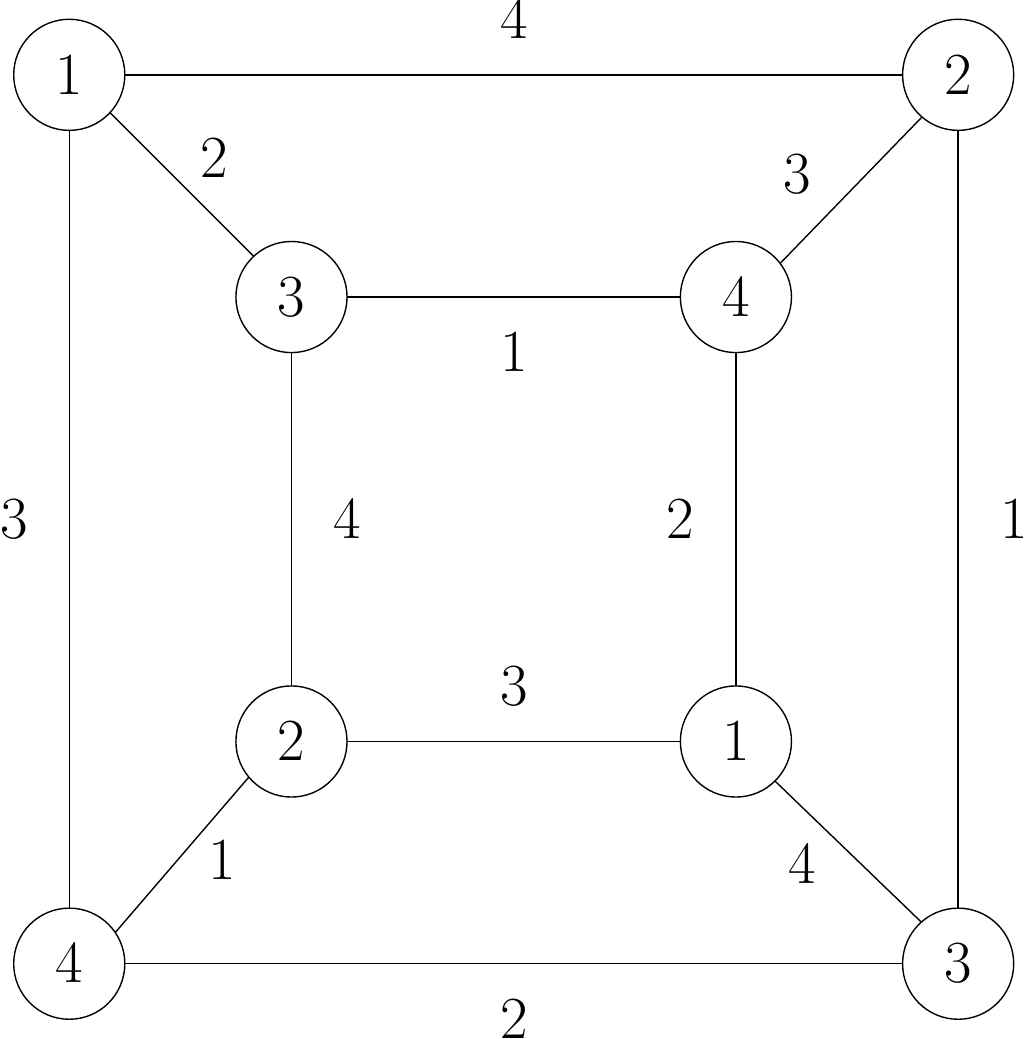}}
\caption{Cube $Q_3$ with star edge-coloring by four colors. The vertex labels are 
used in the proof of Theorem~\ref{thm:cubic}.}
\label{fig:cube}
\end{figure}

There are cubic graphs whose star chromatic index is equal to 6.
One example is $K_{3,3}$. To see this, let us suppose that we have a star
edge-coloring of $K_{3,3}$, and let $F$ be a color class. If $|F|=3$, then every other
color class contains at most one edge and hence there are at least seven
colors all together. So, we may assume that every color class contains one or two 
edges only. If $F = \{ab,cd\}$ is a color class, then one of the edges $ad$ or $cb$
forms a singleton color class since the second edge in the color class of $ad$
(and the same for $cb$) would need to be the edge of $K_{3,3}$ disjoint from
$a,b,c,d$. This implies that there are at least two singleton color classes. 
Hence, the total number of colors is at least 6. Finally, a star 6-edge-coloring
of $K_{3,3}$ is easy to construct, proving that $\chis'(K_{3,3})=6$.

\section{Open problems}

As we saw in Sections~2 and~4, establishing the star 
chromatic index is nontrivial even for complete graphs.
We established bounds 
$$
  (2+o(1))\cdot n \le \chis'(K_n) \le n \cdot \frac{ 2^{ 2\sqrt2(1+o(1)) \sqrt{\log n} } }{(\log n)^{1/4}} \,.
$$
\begin{question}
What is the true order of magnitude of $\chis'(K_n)$? 
Is $\chis'(K_n) = O(n)$? 
\end{question}

In the previous section we obtained the bound
$\chis'(G) \le 7$ for a subcubic graph $G$. 
We also saw that $\chis'(K_{3,3}) = 6$. 
A bipartite cubic graph that we thought might require seven colors is the Heawood graph
(the incidence graph of the points and lines of the Fano plane).
However, it turned out that also its star chromatic index is at most 6. 
After some additional thoughts, we propose the following.

\begin{conjecture}
If\/ $G$ is a subcubic graph, then\/ $\chis'(G)\le 6$.
\end{conjecture}

It would be interesting to understand the list version
of star edge-coloring: by an edge $k$-list for a graph $G$ 
we mean a collection $(L_e)_{e \in E(G)}$ such that 
each $L_e$ is a set of size~$k$. We shall say that $G$~is 
\emph{$k$-star edge choosable} if for every edge $k$-list~$(L_e)$ 
there is a star edge-coloring $c$ such that
$c(e) \in L_e$ for every edge $e$. 
We let $ch'_s(G)$ be the minimum $k$ such that 
$G$ is $k$-star edge choosable.
All of the results in this paper may have extension to list colorings. 
Let us ask specifically two questions: 

\begin{question}
Is it true that $\mathrm{ch}'_s(G) \le 7$ for every subcubic graph~$G$?
(Perhaps even $\le 6$?) 
\end{question}

\begin{question}
Is it true that $\mathrm{ch}'_s(G) = \chi'_s(G)$ for every graph~$G$?
\end{question}

\section*{Acknowledgement} 

We thank to the anonymous referees for helpful comments, in particular 
for suggesting an improvement of the proof of Theorem~\ref{thm:cubic}.

\end{document}